\date{\today}
\documentclass{amsart}
\usepackage{amsmath}
\usepackage{amsfonts}
\usepackage{amscd}
\usepackage{amssymb}
\usepackage{amsthm}
\usepackage{amsopn}
\usepackage{tikz}
\usetikzlibrary{calc}
\usetikzlibrary{patterns}
\usetikzlibrary{matrix}
\usepgflibrary{arrows}
\usetikzlibrary{arrows}
\usetikzlibrary{decorations.pathreplacing}
\usetikzlibrary{decorations.pathmorphing}
\usepackage{epsfig}
\usepackage{enumerate}

\graphicspath{{anticyclegraphics/}}

\DeclareMathOperator{\supp}{supp}

\DeclareMathOperator{\lexgt}{>_{\text lex}}

\DeclareMathOperator{\tor}{Tor}

\newcommand{\mm}{m}
\newcommand{\dual}{I^{\vee}}

\newcommand{\lex}{\text{lex}}
\newcommand{\field}{\Bbbk}
\newcommand{\kk}{\Bbbk}

\theoremstyle{definition}
\newtheorem{thm}{Theorem}[section]
\newtheorem{defn}[thm]{Definition}

\newtheorem{lemma}[thm]{Lemma}

\newtheorem{example}[thm]{Example}
\newtheorem{proposition}[thm]{Proposition}
\newtheorem{remark}[thm]{Remark}

\newtheorem{notation}[thm]{Note}
\newtheorem{question}[thm]{Question}

\title{Linear Quotients of Square of the Edge Ideal of the Anticycle}

\author{A.\,H. Hoefel}
\address{Andrew H. Hoefel\\
Department of Mathematics and Statistics\\
Dalhousie University\\
Halifax, Nova Scotia B3H 4R2}
\email{andrew.hoefel@mathstat.dal.ca}

\author{G. Whieldon}
\address{Gwyneth Whieldon \\
Department of Mathematics\\
Cornell University\\
Ithaca, NY  14850
}
\email{whieldon@math.cornell.edu}

\begin{document}

\let\thefootnote\relax\footnotetext{The first author's research is supported by NSERC and Killam scholarships.}

\maketitle
\begin{abstract}Let $G$ be a graph with chordal complement and $I(G)$ its edge ideal.  From work of Herzog, Hibi, and Zheng, it is known that $I(G)$ has linear quotients and all of its power have linear resolutions.  For edge ideals $I(G)$ arising from graphs which do not have chordal complements, exact conditions on their powers possessing linear resolutions or linear quotients are harder to find.  We provide here an explicit linear quotients ordering for all powers of the edge ideal of the antipath and a linear quotients ordering on the second power $I(A_n)^2$ of the edge ideal of the anticycle $A_n$.  This linear quotients ordering on $I(A_n)^2$ recovers a prior result of Nevo that $I(A_n)^2$ has a linear resolution.
\end{abstract}

\section{Introduction and Background}
Let $G$ be a simple graph on $n$ vertices, and $I(G)$ its edge ideal, 
i.e., a squarefree monomial ideal in $R=\kk[x_1,\ldots,x_n]$ with monomial 
generators $x_ix_j$ corresponding to each edge $\{i,j\}\in G$.  
Such ideals have been extensively studied in such papers as  
\cite{MR2301246}, \cite{Ha:2008:MIE:1344768.1344814}, \cite{MR2739498}, \cite{MR1031197}, 
and more recently, \cite{2010arXiv1012.5329M}.  
A goal of much recent research has been to classify behavior of the resolutions 
of such ideals $I(G)$ and that of their powers in terms of combinatorial data of $G$.  
We provide here an explicit proof that the second power of the edge ideal of the 
anticycle has not just a linear resolution, but also linear quotients.

In the course the proof, we additionally demonstrate that all powers $I(P_n^c)^k$ of the edge ideal of the antipath have linear quotients.
\begin{defn}  
Let $G$ be a simple graph on $n$ vertices. Then the \emph{edge ideal of G} is the squarefree monomial ideal $I(G)$ given by 
\[ I(G) =(x_ix_j:\{i,j\}\in G).
\]
\end{defn}

We say that a graph $G$ has property $P$ if its edge ideal $I(G)$ has such a property; e.g., $G$ is Gorenstein if $I(G)$ is Gorenstein, $G$ is linear if $I(G)$ has a linear resolution, etc. In particular, we will say a graph $G$ has linear quotients if its edge ideal $I(G)$ has linear quotients:

\begin{defn} 
Let $I$ be a homogeneous ideal. We say that \emph{$I$ has linear quotients} if there exists some ordering of the generators of $I=(m_1,m_2,\ldots,m_r)$ such that for all $i>1$,
\[ ((m_1,\ldots,m_{i-1}):(m_i))=(x_{k_1},\ldots,x_{k_s})
\]
for some variables $x_{k_1},\ldots,x_{k_s}$.  We say that such an ordering $(m_1,m_2,\ldots,m_r)$ is a 
\emph{linear quotients ordering of $I$}.
\end{defn}

For two monomials $m$ and $m'$ we define $m' : m$ to be the monomial $\frac{m'}{\gcd(m,m')}$.
Given monomials $m_1, \ldots, m_i$, the colon ideal $(m_1, \ldots, m_{i-1}):(m_i)$ can be computed as
\[ (m_1,\ldots, m_{i-1}):(m_i) = (m_1:m_i, \ldots, m_{i-1}:m_i).
\]
Thus, in order to show that a monomial ideal $I= (m_1, \ldots, m_r)$ has linear quotients, it suffices to show that for each pair of monomials $m_i$ and $m_j$ with $j<i$ that there exists another monomial $m_k$ with $k < i$ with
\[ m_k : m_i = x_l \text{ for some $l$} \qquad \text{and} \qquad x_l \text{ divides } m_j : m_i.
\]

The \emph{graded Betti numbers} of a homogeneous ideal $I$ are given by $\beta_{i,j}(I) = \dim_{\field}\tor_i(I, \field)_j$.  The graded Betti numbers also correspond to the ranks of the free modules in a minimal free resolution of $I$.  We say an ideal $I$ which is generated in degree $d$ has a \emph{linear resolution} if $\beta_{i,j}(I) = 0$ for $j \neq i +d$.  Ideals with linear quotients also have linear resolutions.

Providing a linear quotients ordering is one technique for proving that an ideal has a linear resolution, often with combinatorial significance in the case of monomial ideals.  In the case of squarefree monomial ideal, an ideal $I$ having linear quotients is equivalent to its Alexander dual $\dual$ having a shelling order on its facets. For non-squarefree monomial ideals, a linear quotient orderingscan be viewed as giving a shelling order on the Alexander dual of its polarization.

Interest in powers of the anticycle partially draws from a result of Herzog, Hibi and Zheng \cite{MR2091479} which states the following:
\begin{thm}[Herzog, Hibi, Zheng]\label{thm:HHZ}
Let $I$ be a quadratic monomial ideal of the polynomial ring. The following are equivalent:
\begin{enumerate}
\item $I$ has a linear resolution,
\item $I$ has linear quotients,
\item $I^k$ has a linear resolution for all $k \geq 1$.
\end{enumerate}
\end{thm}

For edge ideals, Fr\"oberg showed that $I(G)$ has a linear resolution if and only if the complement of $G$ is chordal \cite{MR1171260}.

Conspicuously missing from the above theorem is the statement that all powers of a quadratic monomial ideal $I$ with linear resolution must have linear quotients. In fact, this is not known. There are numerous examples of non-quadratic monomial ideals possessing a linear resolution, or even linear quotients, whose powers do not. In \cite{MR2184787}, Conca provides a example generated in degree 3 which is not dependent on the characteristic of the field $\field$.

It would be of interest to construct linear quotients of powers of quadratic monomial ideals with the aim of extending Herzog, Hibi and Zheng's theorem.  Alternately, as no counterexamples are known, the construction of a quadratic monomial ideal $I$ with a linear resolution but some power $k$ with no linear quotients ordering on the generators of $I^k$ would be of combinatorial interest.

Our work on the second power of the anticycle was also inspired by a second thread of research. Francisco, H\`a and Van Tuyl first investigated graphs $G$ where $I(G)^k$ has a linear resolution for each $k \geq 2$. 

From Fr\"oberg and Herzog, Hibi and Zheng's results, we see that chordal graphs have this property. 
More generally, it has been shown by Francisco, H\`a and Van Tuyl that if some power of $I(G)$ has a linear resolution, then the complement of $G$ cannot contain any induced four cycles. Their proof was recorded in \cite{NP2009}.

Inspired by these results, Peeva and Nevo constructed an example of a graph $G$ with no four cycle in its complement and where $I(G)^2$ does not have a linear resolution. Peeva and Nevo have conjectured that their example works only because $I(G)$ has Castelnuovo-Mumford regularity four and that every successive power of an edge ideal should get strictly closer to a linear resolution. See \cite{NP2009} for a more precise statement.

Nevo has also shown that claw-free graphs with no four cycles in their complements have regularity at most three and their second powers have linear resolutions \cite{MR2739498}. Anticycles on more than four vertices meet these criteria and so, it follows that their second powers have linear resolutions.  Here we demonstrate that the square of the edge ideal of the anticycle has linear quotients, recovering this result.

\section{Cycles, Anticycles, and Antipaths}
We first describe the edge ideal of the anticycle and partition pairs of its edges into several natural classes. Next, we provide a linear quotients ordering on these classes relative to the previous generators.

The \emph{complement} of a graph $G$ is the graph on the vertices of $G$ containing all edges that are not in $G$. We use $G^c$ to denote the complement graph.
\begin{defn}  Let $C_n$ be the cycle graph on $n$ vertices, i.e. the graph consisting of one cycle of length $n$ on these vertices with no chords.  The \emph{anticycle graph} $A_n$ is the complement graph of $C_n$, i.e., $A_n = C_n^c$.
\end{defn}

\begin{defn}  The \emph{antipath} $P_n^c$ is the graph on $n$ vertices containing of all edges in the complement of a path $P_n$ of length $n-1$.  We depict the antipath in the figure below.
\end{defn}


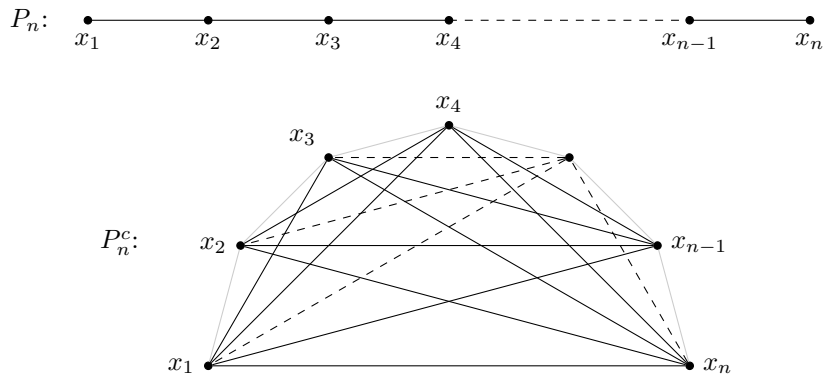
\begin{center}
\begin{tikzpicture}
[scale=.8, vertices/.style={draw, fill=black, circle, inner sep=1pt}]
\useasboundingbox (0,-.75) rectangle (12,.75);
\node [anchor=east] at (-.5,0) {$P_n$:};
\node at (0,0) [vertices, label=below:{$x_1$}]{};
\foreach \to/\from in {1/2,2/3,3/4}
	\draw (-2+2*\to,0)--(-2+2*\from,0)
		node[vertices, label=below:{$x_{\from}$}]{};
\draw [dashed] (2*3,0)--(2*5,0);
\node [vertices, label=below:{$x_{n-1}$}] (5) at (2*5,0) {};
\draw (2*5,0)--(2*6,0)
	node[vertices, label=below:{$x_{n}$}]{};
\end{tikzpicture}
\end{center}

\begin{center}
\begin{tikzpicture}
[scale=.8, vertices/.style={draw, fill=black, circle, inner sep=1pt}]
\useasboundingbox (-5,-.75) rectangle (5,4.75);
\node [anchor=east] at (-5,2.05) {$P_n^c$:};
	\node [vertices, label=right:{$x_n$}] (0) at (0:4) {};
	\node [vertices, label=right:{$x_{n-1}$}] (30) at (30:4) {};
	\node [vertices] (60) at (60:4) {};
	\node [vertices, label=above:{$x_4$}] (90) at (90:4) {};
	\node [vertices, label=above left:{$x_3$}] (120) at (120:4) {};
	\node [vertices, label=left:{$x_2$}] (150) at (150:4) {};
	\node [vertices, label=left:{$x_1$}] (180) at (180:4) {};
\foreach \to/\from in {0/30,90/120,120/150,150/180,30/60,60/90}
	\draw [-, black!20] (\to)--(\from);
\foreach \to/\from in {0/90, 0/120,0/150,0/180,30/90,30/120,30/150,30/180,90/150,90/180,120/180}
	\draw [-] (\to)--(\from);
\foreach \to/\from in {0/60,60/120,60/150,60/180}
	\draw [dashed] (\to)--(\from);
\end{tikzpicture}
\end{center}
Producing a linear quotients ordering for graphs with chordal complements is always possible and all of their powers have linear resolutions, as given in Theorem 3.2 in \cite{MR2091479}.  However, most naive orderings on the generators of higher powers of $I(G)$ fail to produce linear quotients for $G$ with chordal complements.

\begin{example}  Let $R=\kk[x_1,\ldots,x_6]$ and let $I=I(A_n)^2$ be the square of the edge ideal of the anticycle on 6 vertices in $R$.  Its generators, written in lex order, are given by:
\begin{align*}
	&{x}_{1}^{2} {x}_{3}^{2},{x}_{1}^{2} {x}_{3} {x}_{4},{x}_{1}^{2} {x}_{3}
	{x}_{5},{x}_{1}^{2} {x}_{4}^{2},{x}_{1}^{2} {x}_{4} {x}_{5},{x}_{1}^{2}
	{x}_{5}^{2},{x}_{1} {x}_{2} {x}_{3} {x}_{4},{x}_{1} {x}_{2} {x}_{3} {x}_{5},
	{x}_{1}{x}_{2} {x}_{3} {x}_{6},\\
	&{x}_{1} {x}_{2} {x}_{4}^{2}, {x}_{1} {x}_{2} {x}_{4}{x}_{5},{x}_{1} {x}_{2} {x}_{4} 
	{x}_{6},{x}_{1} {x}_{2} {x}_{5}^{2}, {x}_{1} {x}_{2}{x}_{5} {x}_{6},{x}_{1} {x}_{3}^{2}
	{x}_{5},{x}_{1} {x}_{3}^{2} {x}_{6},{x}_{1} {x}_{3}{x}_{4} {x}_{5},\\
	&{x}_{1} {x}_{3} {x}_{4} {x}_{6},
	{x}_{1} {x}_{3} {x}_{5}^{2},{x}_{1}{x}_{3} {x}_{5} {x}_{6}, {x}_{1} {x}_{4}^{2} {x}_{6},
	{x}_{1} {x}_{4} {x}_{5}{x}_{6},{x}_{2}^{2} {x}_{4}^{2},{x}_{2}^{2}
	{x}_{4}{x}_{5},{x}_{2}^{2} {x}_{4} {x}_{6},{x}_{2}^{2} {x}_{5}^{2},\\
	&{x}_{2}^{2} {x}_{5} {x}_{6},
	{x}_{2}^{2}{x}_{6}^{2},
	{x}_{2} {x}_{3} {x}_{4} {x}_{5},{x}_{2} {x}_{3} {x}_{4} {x}_{6},{x}_{2}
	{x}_{3} {x}_{5}^{2},{x}_{2} {x}_{3} {x}_{5} {x}_{6},{x}_{2} {x}_{3} {x}_{6}^{2},
	{x}_{2}{x}_{4}^{2} {x}_{6},\\
	&{x}_{2} {x}_{4} {x}_{5} {x}_{6},{x}_{2} {x}_{4}{x}_{6}^{2},
	{x}_{3}^{2} {x}_{5}^{2},{x}_{3}^{2} {x}_{5} {x}_{6},{x}_{3}^{2} {x}_{6}^{2},{x}_{3} {x}_{4} {x}_{5} {x}_{6},{x}_{3} {x}_{4} {x}_{6}^{2},
	{x}_{4}^{2}{x}_{6}^{2}.
 \end{align*}
This ordering \emph{fails} to be a linear quotients ordering.  Let $m_i$ be the $i^{\text{th}}$ monomial in the ordering above, and let $I_i$ denote the ideal generated by the first $i-1$ monomials in the ordering.  Setting $Q_i=I_i:(m_i)$, we see that 
\begin{align*}
Q_9&=({x}_{1}^{2} {x}_{3}^{2},{x}_{1}^{2} {x}_{3} {x}_{4},{x}_{1} {x}_{2} {x}_{3}
      {x}_{4},{x}_{1}^{2} {x}_{4}^{2},{x}_{1} {x}_{2} {x}_{4}^{2},{x}_{2}^{2}
      {x}_{4}^{2},{x}_{1}^{2} {x}_{3} {x}_{5},{x}_{1} {x}_{2} {x}_{3} {x}_{5}):(x_1x_2x_3x_6)\\
      &=(x_4,x_5,x_1x_3)
\end{align*}
is not generated by variables, hence the lex ordering fails to give us linear quotients.  Similarly, with reverse lex, we have the following ordered generating set:
\begin{align*}
	&{x}_{1}^{2} {x}_{3}^{2},{x}_{1}^{2} {x}_{3} {x}_{4},{x}_{1} {x}_{2} {x}_{3}
	{x}_{4},{x}_{1}^{2} {x}_{4}^{2},{x}_{1} {x}_{2} {x}_{4}^{2},{x}_{2}^{2}
	{x}_{4}^{2},{x}_{1}^{2} {x}_{3} {x}_{5},{x}_{1} {x}_{2} {x}_{3} {x}_{5},{x}_{1}
	{x}_{3}^{2} {x}_{5},\\
	&{x}_{1}^{2} {x}_{4} {x}_{5},{x}_{1} {x}_{2} {x}_{4} {x}_{5},{x}_{2}^{2} {x}_{4} {x}_{5},
	{x}_{1} {x}_{3} {x}_{4} {x}_{5},{x}_{2} {x}_{3}{x}_{4} {x}_{5},{x}_{1}^{2} {x}_{5}^{2},
	{x}_{1} {x}_{2} {x}_{5}^{2},{x}_{2}^{2}{x}_{5}^{2},{x}_{1} {x}_{3} {x}_{5}^{2},\\
	& {x}_{2} {x}_{3} {x}_{5}^{2},{x}_{3}^{2}{x}_{5}^{2}, {x}_{1} {x}_{2} {x}_{3} {x}_{6},
	{x}_{1} {x}_{3}^{2} {x}_{6},{x}_{1} {x}_{2}{x}_{4} {x}_{6},{x}_{2}^{2} {x}_{4} {x}_{6},
	{x}_{1} {x}_{3} {x}_{4} {x}_{6}, {x}_{2} {x}_{3} {x}_{4} {x}_{6},\\
	& {x}_{1} {x}_{4}^{2} {x}_{6},{x}_{2} {x}_{4}^{2} {x}_{6},
	{x}_{1}{x}_{2} {x}_{5} {x}_{6},{x}_{2}^{2} {x}_{5} {x}_{6},{x}_{1} {x}_{3} {x}_{5}
	{x}_{6},{x}_{2} {x}_{3} {x}_{5} {x}_{6},{x}_{3}^{2} {x}_{5} {x}_{6},{x}_{1} {x}_{4} {x}_{5} {x}_{6},\\
	& {x}_{2} {x}_{4} {x}_{5} {x}_{6},{x}_{3} {x}_{4} {x}_{5} {x}_{6}, {x}_{2}^{2} {x}_{6}^{2},
	{x}_{2} {x}_{3} {x}_{6}^{2},{x}_{3}^{2}{x}_{6}^{2},{x}_{2} {x}_{4} {x}_{6}^{2},
	{x}_{3} {x}_{4} {x}_{6}^{2},{x}_{4}^{2} {x}_{6}^{2}.
\end{align*}
This fails to have linear quotients at $Q_{21}=I_{21}:(x_1x_2x_3x_6)=(x_4,x_5,x_1x_3)$.  Using a monomial ordering on the generators of $I$ does not appear to ever produce a linear quotients ordering on the generators of $I(A_n)^2$.
\end{example}
This appears to be true more generally -- while all higher powers of edge ideals with linear quotients appear to have linear quotients as well, these linear quotients orders almost never arise from a monomial term ordering.
\section{Antipath Linear Quotients}
Throughout this section we will use $H=P_n^c$ to denote the antipath on $n$ vertices.
The first stage in our linear quotients ordering is to show that the square of the antipath has linear quotients with respect to the lex order.  As the complement of the antipath is a chordal graph, it is known that $I(H)$ has a linear resolution via Fr\"{o}berg's Theorem \cite{MR1171260}. Furthermore, as $I(H)$ has a linear resolution and is generated in degree 2, it is known to have a linear quotient ordering and linear resolutions of all of its powers \cite{MR2091479}. However, a linear resolution of its second power does not guarantee a linear quotients ordering of $I(H)^k$, which we provide explicitly here.
\begin{proposition}\label{prop:kantipath} 
The $k^{\text{th}}$ power $I(H)^k$ of the edge ideal of the antipath $H$ has linear quotients, 
under the lex ordering of the generators.
\end{proposition}
We begin with some notation and a lemma.

Given any $k$ edges $e_1, \ldots, e_k$ in a graph $G$, we will often abuse notation and
write $\mm=e_1e_2\cdots e_k$ for the monomial
\[ \mm=\prod_{r=1}^k x_{i_r} x_{j_r}
\]
where $e_r=\{x_{i_r},x_{j_r}\}$. When a monomial $\mm$ is of this form, we say $\mm$ is the \emph{product of $k$ edges} of $G$. 
\begin{example}Let $G$ be the complete graph on six vertices $\{x,y,z,w,s,t\}$ seen below.

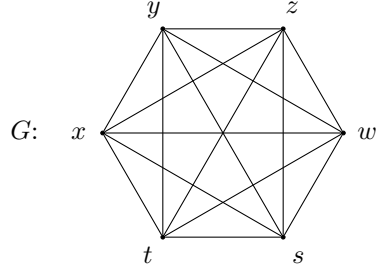
\begin{center}
\begin{tikzpicture}
[scale=.8, vertices/.style={draw, fill=black, circle, inner sep=0.5pt}]
\useasboundingbox (-2.5,-2.4) rectangle (2.5,2.4);
\node [anchor=base] at (-3.3,-.1){$G$:};
\node [vertices] (1) at (0:2) {};
\node [anchor=base] at (2.4,-.1) {$w$};
\node [vertices] (2) at (60:2) {};
\node [anchor=base] at (60:2.3) {$z$};
\node [vertices] (3) at (120:2) {};
\node [anchor=base] at (120:2.3) {$y$};
\node [vertices] (4) at (180:2) {};
\node [anchor=base] at (-2.4, -.1) {$x$};
\node [vertices] (5) at (240:2) {};
\node [anchor=base] at (240:2.5) {$t$};
\node [vertices] (6) at (300:2) {};
\node [anchor=base] at (300:2.5) {$s$};
\foreach \to/\from in {2/1, 3/1, 3/2, 4/1, 4/2, 4/3, 5/1, 5/2, 5/3, 5/4, 6/1, 6/2, 6/3, 6/4, 6/5}
	\draw [-] (\to)--(\from);
\end{tikzpicture}
\end{center}
Then the monomial $\mm=xyzwst\in I(G)^3$ comes from any three edges with each vertex appearing in a unique edge exactly once.
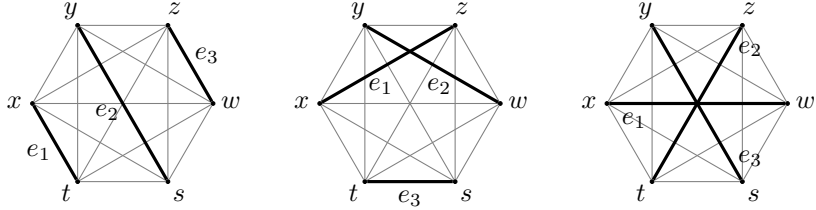
\begin{center}
\begin{tikzpicture}
[scale=.6, vertices/.style={draw, fill=black, circle, inner sep=0.5pt}]
\useasboundingbox (-2.5,-2.5) rectangle (2.5,2.5);
\node [vertices] (1) at (0:2) {};
\node [anchor=base] at (2.4,-.1) {$w$};
\node [vertices] (2) at (60:2) {};
\node [anchor=base] at (60:2.3) {$z$};
\node [vertices] (3) at (120:2) {};
\node [anchor=base] at (120:2.3) {$y$};
\node [vertices] (4) at (180:2) {};
\node [anchor=base] at (-2.4, -.1) {$x$};
\node [vertices] (5) at (240:2) {};
\node [anchor=base] at (240:2.5) {$t$};
\node [vertices] (6) at (300:2) {};
\node [anchor=base] at (300:2.5) {$s$};
\foreach \to/\from in {3/1, 3/2, 4/1, 4/2, 4/3, 5/1, 5/2, 5/3, 6/1, 6/2, 6/4, 6/5}
	\draw [black!50,-] (\to)--(\from);
\foreach \to/\from in {5/4}
	\draw [very thick, -] (\to)--(\from);
\foreach \to/\from in {6/3}
	\draw [very thick, -] (\to)--(\from);
\foreach \to/\from in {1/2}
	\draw [very thick, -] (\to)--(\from);
\node [anchor=center] at (210:2.15) {$e_1$};
\node [anchor=center] at (210:.4) {$e_2$};
\node [anchor=center] at (30:2.15) {$e_3$};
\end{tikzpicture}\hspace{2em}
\begin{tikzpicture}
[scale=.6, vertices/.style={draw, fill=black, circle, inner sep=0.5pt}]
\useasboundingbox (-2.5,-2.5) rectangle (2.5,2.5);
\node [vertices] (1) at (0:2) {};
\node [anchor=base] at (2.4,-.1) {$w$};
\node [vertices] (2) at (60:2) {};
\node [anchor=base] at (60:2.3) {$z$};
\node [vertices] (3) at (120:2) {};
\node [anchor=base] at (120:2.3) {$y$};
\node [vertices] (4) at (180:2) {};
\node [anchor=base] at (-2.4, -.1) {$x$};
\node [vertices] (5) at (240:2) {};
\node [anchor=base] at (240:2.5) {$t$};
\node [vertices] (6) at (300:2) {};
\node [anchor=base] at (300:2.5) {$s$};
\foreach \to/\from in {2/1, 3/2, 4/1, 4/3, 5/1, 5/2, 5/3, 5/4, 6/1, 6/2, 6/3, 6/4}
	\draw [black!50,-] (\to)--(\from);
\foreach \to/\from in {1/3}
	\draw [very thick, -] (\to)--(\from);
\foreach \to/\from in {4/2}
	\draw [very thick, -] (\to)--(\from);
\foreach \to/\from in {5/6}
	\draw [very thick, -] (\to)--(\from);
\node [anchor=center] at (150:.75) {$e_1$};
\node [anchor=center] at (30:.75) {$e_2$};
\node [anchor=center] at (270:2.1) {$e_3$};
\end{tikzpicture}\hspace{2em}
\begin{tikzpicture}
[scale=.6, vertices/.style={draw, fill=black, circle, inner sep=0.5pt}]
\useasboundingbox (-2.5,-2.5) rectangle (2.5,2.5);
\node [vertices] (1) at (0:2) {};
\node [anchor=base] at (2.4,-.1) {$w$};
\node [vertices] (2) at (60:2) {};
\node [anchor=base] at (60:2.3) {$z$};
\node [vertices] (3) at (120:2) {};
\node [anchor=base] at (120:2.3) {$y$};
\node [vertices] (4) at (180:2) {};
\node [anchor=base] at (-2.4, -.1) {$x$};
\node [vertices] (5) at (240:2) {};
\node [anchor=base] at (240:2.5) {$t$};
\node [vertices] (6) at (300:2) {};
\node [anchor=base] at (300:2.5) {$s$};
\foreach \to/\from in {2/1, 3/1, 3/2, 4/2, 4/3, 5/1, 5/3, 5/4, 6/1, 6/2, 6/4, 6/5}
	\draw [black!50,-] (\to)--(\from);
\foreach \to/\from in {4/1}
	\draw [very thick, -] (\to)--(\from);
\foreach \to/\from in {2/5}
	\draw [very thick, -] (\to)--(\from);
\foreach \to/\from in {3/6}
	\draw [very thick, -] (\to)--(\from);
\node [anchor=north] at (180:1.4) {$e_1$};
\node [anchor=west] at (60:1.4) {$e_2$};
\node [anchor=west] at (300:1.4) {$e_3$};
\end{tikzpicture}
\end{center}
So $\mm=e_1e_2e_3$ for the labeled edge sets in any of the diagrams above.
\end{example}
\begin{lemma}\label{lem:kpowers}  The ideal $I(H)^k$ is given by all monomials of degree $2k$ of the form
\begin{multline*}
I(H)^k=(x_{i_1}x_{i_2}\cdots x_{i_k}x_{j_1}x_{j_2}\cdots x_{j_k}: \\
i_1\leq i_2\leq\cdots\leq i_k\leq j_1\leq j_2\leq \cdots j_k \text{ and }i_r+2\leq j_r \text{ for all }r).
\end{multline*}
\end{lemma}
Equivalently, every minimal monomial generator $\mm \in I(H)^k$ can be written as
a product of $k$ edges $\mm = e_1\cdots e_k$ where $e_r=\{x_{i_r},x_{j_r}\}$
and 
\[ i_1\leq i_2\leq\cdots \leq i_k\leq j_1\leq j_2\leq \cdots \leq j_k.
\]
\begin{proof}
Any monomial $\mm$ of degree $2k$ can be written as
\[ \mm = x_{i_1} \cdots x_{i_k} x_{j_1} \cdots x_{j_k}
\]
with $i_1 \leq \cdots \leq i_k \leq j_1 \leq \cdots \leq j_k$. Let $\mm$ be a minimal generator of $I(H)^k$ and
write $\mm$ as above. Assume for a contradiction that there is an index $r$ with $i_r + 2 > j_r$. Since the 
indices of $\mm$ have been written in ascending order, we know that 
\[ \{ i_r, i_{r+1}, \ldots, i_k, j_1, \ldots, j_r \} \subseteq \{ i_r, i_r +1\}.
\]

Let $\mm'$ be the degree $k+1$ monomial $\mm' = x_{i_r}\cdots x_{i_k}x_{j_1} \cdots x_{j_r}$ which 
divides $\mm$. The support of $\mm'$ is contained in $\{x_{i_r}, x_{i_r +1}\}$ but there are 
are no edges in the antipath between $x_{i_r}$ and $x_{i_r +1}$. Thus, $\mm'$ contains no edge as a factor.
However, as $\mm$ is a product of $k$ edges, every degree $k+1$ factor of $\mm$ must contain at least one edge.
This is contradicted by our construction of $\mm'$, and so we must have $i_r + 2 \leq j_r$ for each $r$.
\end{proof}

We now return to the proof of Proposition~\ref{prop:kantipath}.
\begin{proof}[Proof of Proposition~\ref{prop:kantipath}]  From Lemma~\ref{lem:kpowers}, we have that
\begin{multline*}
I(H)^k=(x_{i_1}x_{i_2}\cdots x_{i_k}x_{j_1}x_{j_2}\cdots x_{j_k}: \\
i_1\leq i_2\leq\cdots\leq i_k\leq j_1\leq j_2\leq \cdots j_k \text{ and }i_r+2\leq j_r \text{ for all }r).
\end{multline*}
Any pair of monomial generators $\mm$ and $\mm'$ of $I(H)^k$ will be of the forms:
\begin{align*}
\mm=x_{i_1}x_{i_2}\cdots x_{i_k}x_{j_1}x_{j_2}\cdots x_{j_k}=e_{1}e_{2}\cdots e_k\\
\mm'=x_{i_1'}x_{i_2'}\cdots x_{i_k'}x_{j_1'}x_{j_2'}\cdots x_{j_k'}=e_1'e_2'\cdots e_k'
\end{align*}
with indices $i_r,i_r',j_r,j_r'$ all satisfying the inequalities above 
and for edges $e_r = \{ x_{i_r}, x_{j_r}\}$ and $e'_r = \{x_{i'_r}, x_{j'_r}\}$ of $H$. We show for every such pair of monomials
 with $\mm' \lexgt \mm$ that $\mm' : \mm$ will be divisible by some $x_i=\mm'':\mm$ for some $\mm''\lexgt \mm$.

{\bf Case 1:  Monomials $\mm$ and $\mm'$ differ first at some $x_{i_r}$.}  
Assume $i_r$ is the first index at which $\mm$ and $\mm'$ differ; i.e., $i_s=i_s'$ for all $s<r$ and $i_r'<i_r$.

Let $\mm''=\frac{\displaystyle x_{i_r'}}{\displaystyle x_{i_r}}\mm.$  This is certainly a monomial of the appropriate degree which is lex earlier than $\mm$.  To show that $\mm''\in I(H)^k$, we note that as $i_r'<i_r<j_r-2$, we have an edge $\varepsilon_r=\{x_{i_r'},x_{j_r}\}\in H$. Thus
$$\mm''=e_1\cdots e_{r-1}\varepsilon_r e_{r+1}\cdots e_k\in I(H)^k.$$
As $\mm'':\mm=x_{i_r'}$ and $x_{i_r'}$ divides $\mm':\mm$, we either had $\mm''=\mm'$ (in which case we satisfy the first condition above) or $\mm''\neq \mm'$ and this colon satisfies the second condition above.

{\bf Case 2: Monomials $\mm$ and $\mm'$ differ first at some $x_{j_r}$.}  
Assume that $\mm$ and $\mm'$ do not differ in the $x_{i_s}$; i.e., $i_s=i_s'$ for all $s=1,\ldots,k$. 
Let $j_r$ be the first index where $\mm$ and $\mm'$ differ. That is, $j_s = j_s'$ for all $s<r$ and $j_r'<j_r$.  So
\begin{align*}
\mm=x_{i_1}\cdots x_{i_k}x_{j_1}\cdots x_{j_{r-1}}x_{j_r}x_{j_{r+1}}\cdots x_{j_k}=e_1e_2\cdots e_{r-1}e_r e_{r+1}\cdots e_k\\
\mm'=x_{i_1}\cdots x_{i_k}x_{j_1}\cdots x_{j_{r-1}}x_{j_r'}x_{j_{r+1}'}\cdots x_{j_k'}=e_1e_2\cdots e_{r-1}e_r' e_{r+1}'\cdots e_k'.
\end{align*}
Choosing
\begin{align*}
\mm''&=x_{i_1}\cdots x_{i_k}x_{j_1}\cdots x_{j_{r-1}}x_{j_r'}x_{j_{r+1}}\cdots x_{j_k}\\
&=e_1e_2\cdots e_{r-1}e_r' e_{r+1}\cdots e_r,
\end{align*}
we note that as $e_r'=\{x_{i_r},x_{j_r'}\}\in H$, we have $\mm''\in I(H)^k$.  This is a lex earlier monomial in $I(H)^k$.  So $\mm'':\mm=x_{j_r'}$ which divides $\mm':\mm$.
\end{proof}

\section{Linear Quotient Ordering of Anticycle}

The proof that the square of the edge ideal of the antipath has linear quotients is the first step in constructing a linear quotients ordering of the generators of the anticycle.  With this in hand, we now show that the following ordering on the generators of the square of the edge ideal of the anticycle gives us linear quotients.  For the remainder of this note, we let $G$ be the anticycle graph and let $H$ be the antipath obtained by deleting some vertex of $G$.

\begin{remark}\label{anticyclelabels}
We will label the vertices in $G$ as follows. Let $x$ be the vertex we delete to obtain $H$, and let $z_1$ and $z_2$ the two non-adjacent vertices in $G$ (so the two neighbors of $x$ in the cycle itself). Finally, let $y_1,\ldots,y_n$ be all the remainging vertices in order, so that $y_1$ is not adjacent to $z_1$ and $y_n$ is not adjacent to $z_n$. Note that each $y_i$ is adjacent to $x$.
Thus, for this section, we assume that $G$ has $n+3$ vertices. See the figure below.
\end{remark}


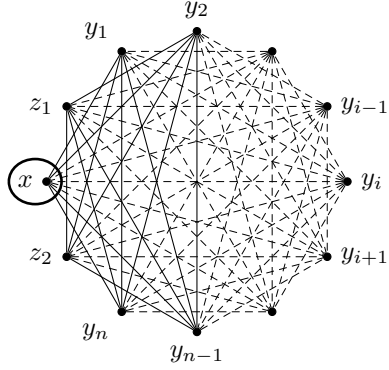
\begin{center}
\begin{tikzpicture}
[scale=.50, vertices/.style={draw, fill=black, circle, inner sep=1pt}]
	\node [vertices, label=right:{$y_i$}] (0) at (0:4) {};
	\node [vertices, label=right:{$y_{i-1}$}] (30) at (30:4) {};
	\node [vertices, label=above right:{}] (60) at (60:4) {};
	\node [vertices, label=above:{$y_2$}] (90) at (90:4) {};
	\node [vertices, label=above left:{$y_1$}] (120) at (120:4) {};
	\node [vertices, label=left:{$z_1$}] (150) at (150:4) {};
	\node [vertices, label=left:{$x$}] (180) at (180:4) {};
	\node [vertices, label=left:{$z_2$}] (210) at (210:4){};
	\node [vertices, label=below left:{$y_n$}] (240) at (240:4){};
	\node [vertices, label=below:{$y_{n-1}$}] (270) at (270:4){};
	\node [vertices, label=left:{}] (300) at (300:4){};
	\node [vertices, label=right:{$y_{i+1}$}] (330) at (330:4){};

\foreach \to/\from in {90/150,90/180,120/180,180/240,180/270,90/210,90/240,90/270,120/210,120/240,120/270, 150/210,150/240,150/270,210/270}
	\draw [-] (\to)--(\from);
\foreach \to/\from in {0/60,30/90,30/120,30/150,30/180,0/90, 0/120,0/150,0/180, 0/210,0/240,0/270,0/300,30/210,30/240,30/270,30/300,30/330,60/150,60/180,60/210,60/240,60/270,60/300,60/330,60/120,270/330,90/330,90/300,120/330,120/300,150/300,150/330,180/330,180/300,210/330,210/300,240/300,240/330}
	\draw [densely dashed, very thin] (\to)--(\from);

\draw [line width=1pt](-5,0) arc (180:-180:.7 and .6);
\end{tikzpicture}
\end{center}

\begin{thm}\label{thm:mainanticycletheorem} 
Let $G$ be the $(n+3)$-anticycle graph, labeled as in the picture above, with $n \geq 2$. 
Let $H=G\setminus\{x\}$ be the induced graph away from $x$. 
Let $J=I(H)$ be the edge ideal of $H$ and let $K=I(G\setminus H)=(xy_i: i=1,\ldots,n)$ be the edge ideal
on the edges not in $H$.

Then the edge ideal $I(G)$ has a linear quotients given by the following ordering of its monoimal generators (monomials occurring earlier in this list appear earlier in the order):
\begin{enumerate}
\item\label{Jsquared} $\mm\in J^2$ ordered via the lex ordering with $z_1< y_1<y_2<\cdots <y_n<z_2$
\item\label{JK} $\mm\in J\cdot K$
	\begin{enumerate}
	\renewcommand{\theenumi}{\alph{enumi}}
\renewcommand{\labelenumi}{(\ref{JK}\theenumi)}
	\item\label{zees}$\mm = xy_iz_1z_2$, $i=1,\ldots,n$,
	\item\label{z2}$\mm = xy_iy_jz_2$, $i\leq j$, ordered via lex with $y_1>y_2>\cdots>y_n$, excluding nongenerator $xy_n^2z_2$,
	\item\label{z1}$\mm = xy_iy_jz_1$, $i\leq j$, ordered via lex with $y_1<y_2<\cdots<y_n$, excluding nongenerator $xy_1^2z_1$, and
	\item\label{whys} $\mm = xy_iy_jy_k$, $i\leq j\leq k$, ordered via lex with $y_1>y_2>\cdots>y_n$.
	\end{enumerate}
\item\label{Ksquared} $\mm\in K^2$.
\begin{enumerate}
\item\label{not1}$\mm=x^2y_iy_j$ ordered via lex excluding $x^2y_1^2$ with $y_1<y_2<\cdots <y_n$
\item $\mm\label{last1}=x^2 y_1^2$.
\end{enumerate}
\end{enumerate}
\end{thm}


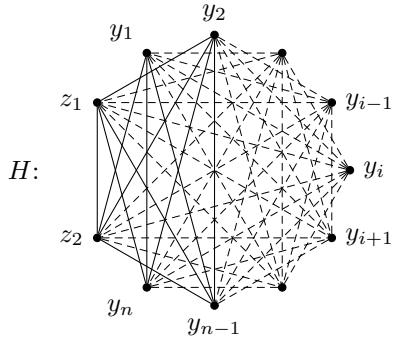
\begin{center}
\begin{tikzpicture}
[scale=.45, vertices/.style={draw, fill=black, circle, inner sep=1pt}]
	\node [anchor=east] at (-4.9,0) {$H$:};
	\node [vertices, label=right:{$y_i$}] (0) at (0:4) {};
	\node [vertices, label=right:{$y_{i-1}$}] (30) at (30:4) {};
	\node [vertices, label=above right:{}] (60) at (60:4) {};
	\node [vertices, label=above:{$y_2$}] (90) at (90:4) {};
	\node [vertices, label=above left:{$y_1$}] (120) at (120:4) {};
	\node [vertices, label=left:{$z_1$}] (150) at (150:4) {};
	\node [vertices, label=left:{$z_2$}] (210) at (210:4){};
	\node [vertices, label=below left:{$y_n$}] (240) at (240:4){};
	\node [vertices, label=below:{$y_{n-1}$}] (270) at (270:4){};
	\node [vertices, label=left:{}] (300) at (300:4){};
	\node [vertices, label=right:{$y_{i+1}$}] (330) at (330:4){};

\foreach \to/\from in {90/150,90/210,90/240,90/270,120/210,120/240,120/270, 150/210,150/240,150/270,210/270}
	\draw [-] (\to)--(\from);
\foreach \to/\from in {0/60,30/90,30/120,30/150,0/90, 0/120,0/150, 0/210,0/240,0/270,0/300,30/210,30/240,30/270,30/300,30/330,60/150,60/210,60/240,60/270,60/300,60/330,60/120,270/330,90/330,90/300,120/330,120/300,150/300,150/330,210/330,210/300,240/300,240/330}
	\draw [densely dashed, very thin] (\to)--(\from);

\end{tikzpicture}\hfill
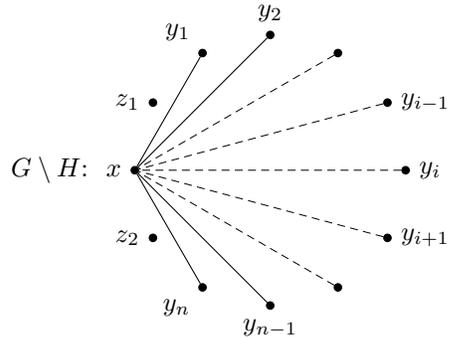
\begin{tikzpicture}
[scale=.45, vertices/.style={draw, fill=black, circle, inner sep=1pt}]
	\node [anchor=east] at (-5.1,0) {$G \setminus H$:};
	\node [vertices, label=right:{$y_i$}] (0) at (0:4) {};
	\node [vertices, label=right:{$y_{i-1}$}] (30) at (30:4) {};
	\node [vertices, label=above right:{}] (60) at (60:4) {};
	\node [vertices, label=above:{$y_2$}] (90) at (90:4) {};
	\node [vertices, label=above left:{$y_1$}] (120) at (120:4) {};
	\node [vertices, label=left:{$z_1$}] (150) at (150:4) {};
	\node [vertices, label=left:{$x$}] (180) at (180:4) {};
	\node [vertices, label=left:{$z_2$}] (210) at (210:4){};
	\node [vertices, label=below left:{$y_n$}] (240) at (240:4){};
	\node [vertices, label=below:{$y_{n-1}$}] (270) at (270:4){};
	\node [vertices, label=left:{}] (300) at (300:4){};
	\node [vertices, label=right:{$y_{i+1}$}] (330) at (330:4){};

\foreach \to/\from in {90/180,120/180,180/240,180/270}
	\draw [-] (\to)--(\from);
\foreach \to/\from in {0/180,30/180,60/180,300/180,330/180}
	\draw [densely dashed, very thin] (\to)--(\from);

\end{tikzpicture}
\end{center}

\noindent Before giving the proof, we provide a specific example of the ordering of $I(G)^2$ for the antipath $G$ on $6$ vertices.
\begin{example}
Let $n=3$ so we have the anticycle graph $G$ on vertices $\{x,z_1,y_1,y_2,y_3,z_2\}$.
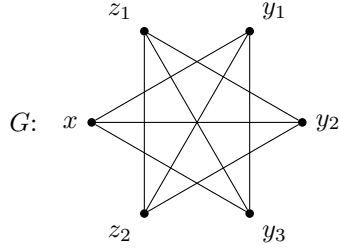
\begin{center}
\begin{tikzpicture}
[scale=.35, vertices/.style={draw, fill=black, circle, inner sep=1pt}]
	\useasboundingbox (-6,-5) rectangle (6,5);
	\node [anchor=east] at (-5.7,0) {$G$:};
	\node [vertices, label=right:{$y_2$}] (0) at (0:4) {};
	\node [vertices, label=above right:{$y_1$}] (60) at (60:4) {};
	\node [vertices, label=above left:{$z_1$}] (120) at (120:4) {};
	\node [vertices, label=left:{$x$}] (180) at (180:4) {};
	\node [vertices, label=below left:{$z_2$}] (240) at (240:4) {};
	\node [vertices, label=below right:{$y_3$}] (300) at (300:4) {};
\foreach \to/\from in {0/120,0/180,0/240,60/180,60/240,60/300,120/240,120/300,180/300}
	\draw [-] (\to)--(\from);
\end{tikzpicture}
\end{center}
Our two subgraphs $H$ and $G\setminus H$ will be as below.

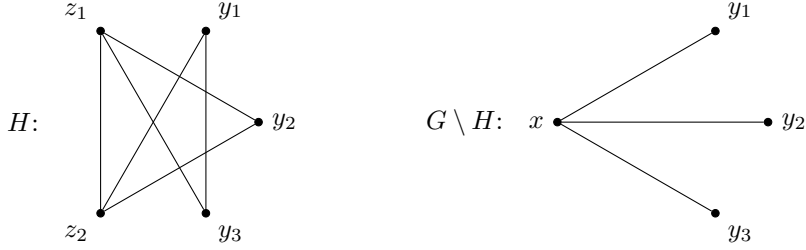
\begin{center}
\begin{tikzpicture}
[scale=.35, vertices/.style={draw, fill=black, circle, inner sep=1pt}]
	\useasboundingbox (-8,-5) rectangle (8,5);
	\node [anchor=east] at (-4,0) {$H$:};
	\node [vertices, label=right:{$y_2$}] (0) at (0:4) {};
	\node [vertices, label=above right:{$y_1$}] (60) at (60:4) {};
	\node [vertices, label=above left:{$z_1$}] (120) at (120:4) {};
	\node [vertices, label=below left:{$z_2$}] (240) at (240:4) {};
	\node [vertices, label=below right:{$y_3$}] (300) at (300:4) {};
\foreach \to/\from in {0/120,0/240,60/240,60/300,120/240,120/300}
	\draw [-] (\to)--(\from);
\end{tikzpicture}\hspace{3em}
\begin{tikzpicture}
[scale=.35, vertices/.style={draw, fill=black, circle, inner sep=1pt}]
	\useasboundingbox (-8,-5) rectangle (8,5);
	\node [anchor=east] at (-5.7,0) {$G \setminus H$:};
	\node [vertices, label=right:{$y_2$}] (0) at (0:4) {};
	\node [vertices, label=above right:{$y_1$}] (60) at (60:4) {};
	\node [vertices, label=left:{$x$}] (180) at (180:4) {};
	\node [vertices, label=below right:{$y_3$}] (300) at (300:4) {};
\foreach \to/\from in {0/180,60/180,180/300}
	\draw [-] (\to)--(\from);
\end{tikzpicture}
\end{center}
The linear quotients ordering from Theorem~\ref{thm:mainanticycletheorem} on the generators of $I(G)^2$ is given here by
\begin{align*}
I(G)^2=&\,\phantom{+}\,  
	( {z}_{1}^{2} {y}_{2}^{2},{z}_{1}^{2} {y}_{2} {y}_{3},{z}_{1}^{2} {y}_{2}
	{z}_{2},{z}_{1}^{2} {y}_{3}^{2}, {z}_{1}^{2} {y}_{3} {z}_{2},{z}_{1}^{2}
	{z}_{2}^{2},{z}_{1} {y}_{1} {y}_{2} {y}_{3},{z}_{1} {y}_{1} {y}_{2} {z}_{2},\\
	&\,\phantom{+}\phantom{(}\, 
	{z}_{1} {y}_{1} {y}_{3}^{2},{z}_{1} {y}_{1} {y}_{3} {z}_{2},
	{z}_{1} {y}_{1} {z}_{2}^{2},{z}_{1}{y}_{2}^{2} {z}_{2},
	{z}_{1} {y}_{2} {y}_{3} {z}_{2},{z}_{1} {y}_{2} {z}_{2}^{2},{y}_{1}^{2} {y}_{3}^{2},\\
	&\,\phantom{+}\phantom{(}\, 
	{y}_{1}^{2} {y}_{3} {z}_{2},{y}_{1}^{2} {z}_{2}^{2},
	{y}_{1} {y}_{2} {y}_{3} {z}_{2}, {y}_{1} {y}_{2} {z}_{2}^{2},{y}_{2}^{2} {z}_{2}^{2})^{(\ref{Jsquared})}\\
	&+(x {z}_{1} {y}_{1} {z}_{2},x {z}_{1} {y}_{2} {z}_{2},x {z}_{1} {y}_{3} {z}_{2})^{(2a)}\\
	&+(x {y}_{1}^{2} {z}_{2},x {y}_{1} {y}_{2} {z}_{2},x {y}_{1} {y}_{3} {z}_{2},x
	{y}_{2}^{2} {z}_{2},x {y}_{2} {y}_{3} {z}_{2})^{(2b)}\\
	&+(x {z}_{1} {y}_{3}^{2},x {z}_{1} {y}_{2} {y}_{3}, x {z}_{1} {y}_{1} {y}_{3}, 
	x {z}_{1} {y}_{2}^{2}, x {z}_{1} {y}_{1} {y}_{2})^{(2c)}\\
	&+(x {y}_{1}^{2} {y}_{3},x {y}_{1} {y}_{2} {y}_{3},x {y}_{1} {y}_{3}^{2})^{(2d)}\\
	&+(x^{2} {y}_{1} {y}_{2},x^{2} {y}_{1} {y}_{3},x^{2} {y}_{2}^{2},x^{2} {y}_{2}
	{y}_{3},x^{2} {y}_{3}^{2})^{(\ref{not1})}\\
	&+(x^{2} {y}_{1}^{2})^{(\ref{last1})}.
\end{align*}
\end{example}

\subsection{Proof of Theorem~\ref{thm:mainanticycletheorem}}
\begin{proof}[Proof of Theorem~\ref{thm:mainanticycletheorem}]The generators of $I(G)^2$ fall into three main cases, with the second case split up into four subcases and the third case placing the first lex ordered generator at the very end.  We will address each case separately.
\begin{notation}Let $I_{M} = \left(I(G)^2\right)_M$ denote the ideal generated by all monomials in the linear quotients ordering before adding $M$, a minimal generator of $I(G)^2$.  In general, we will use $Q_{M}$ to denote the colon ideal
$$Q_{M}=I_{M}:(M),$$
though we will often omit the subscript if the stage in the ordering is clear.  We show here for all monomial generators $M$ in the quotients ordering that
$$Q_{M}=(x_{i_1},x_{i_2},\ldots,x_{i_k})$$
for some variables $x_{i_1},x_{i_2},\ldots,x_{i_k}\in\{x,z_1,z_2,y_1,y_2,\ldots,y_n\}=V$.

Let $V_{M}$ denote the variables generating $Q_{M}$, or as above, $V_{M}=\{x_{i_1},x_{i_2},\ldots,x_{i_k}\}$ and let $W_{M}=V\setminus V_{M}$.

The general technique used begins with generating $x_i\in V_{M}$ explicitly via exhibition of a monomial generator $\mm'\in I_{M}$ such that
$$\mm':M=x_i.$$
After finding our expected $V_{M}$, we note that any remaining minimal monomial generators $\mm$ of $Q_{M}$ which are not variables, i.e. not in a linear generator of the ideal $(V_{M})$, must have their support, $\supp(\mm)\in W_{M}$.

We then show that any generators $\mm'\in I(G)^2$ which would give us
$$\mm':M=\mm\in(W_{M})$$
must either have $\mm\in(V_M)$ (and hence a contradiction, as such a generator cannot be minimal in $Q_M$) or could only come from a monomial $\mm'$ occurring after $M$ in the linear quotients ordering (and hence another contradiction, as $\mm\not\in Q_M$.)  For consistency, we will always use $M$, $\mm$ and $\mm'$ in the same roles throughout the proof.
\end{notation}
\subsubsection{Stage (1):}  Note that $I(H)$ is the antipath graph of the path $\{z_1\sim y_1\sim y_2\sim\cdots \sim y_n\sim z_2\}$, so the ordering of $J^2$ given in (\ref{Jsquared}) is a linear quotients ordering by Proposition~\ref{prop:kantipath}.
\subsubsection{Stage (2a):}  We now move on to generators in $(2a)$ and show that after adding through the $(i-1)^{\text{st}}$ term in $(2a)$, we have linear quotients when we colon this ideal against our ${\text i}^{\text{th}}$ term, $M=z_1z_2x y_i$.  Let $Q$ be this colon ideal,
\begin{align*}
Q &= I_{z_1z_2xy_i}:(z_1z_2xy_i)\\
&=(J^2 + (z_1z_2 xy_j \mid 1 \leq j\leq i-1) ) : (z_1 z_2 x y_i).
\end{align*}
Note that the following inclusions hold, via the elements noted on the right.
	\begin{itemize}
	\item $Q \supseteq (y_j \mid j \neq i)$ as $y_j = z_1 z_2 y_j y_i : z_1 z_2 x y_i$.
	\item $Q \supseteq (z_1)$ when $i \neq 1$ as $z_1 = z_1 z_2 z_1 y_i : z_1 z_2 x y_i$.
	\item $Q \supseteq (z_2)$ when $i \neq n$ as $z_2 = z_1 z_2 z_2 y_i : z_1 z_2 x y_i$.
	\item $Q \supseteq (y_i)$ when $i \not\in \{1, n\}$ as $y_i = y_i^2 z_1 z_2 : z_1 z_2 x y_i$.
	\end{itemize}

\noindent Assume $\mm \in Q$ is a minimal monomial generator of $Q$ that is not linear, i.e. $\mm = \mm' : z_1z_2 x y_i$ for some $\mm'$ appearing in the ordering earlier than $z_1z_2 x y_i$. As $\mm$ is minimal, its support cannot contain any of the variables in $Q$ and therefore

	\[	\supp(\mm) \subseteq \begin{cases}
					\{x\}	& i = 2,\ldots,n-1 , \\
					\{x,z_1,y_1\} & i = 1,\\
					\{x,z_2,y_n\} & i = n.
				\end{cases}
	\]
\noindent In the first of these cases, we note that if $x | \mm$ then $x^2 | \mm'$.  As this does not happen for any $\mm'$ before $z_1z_2x y_i$, the only cases we need to consider are $i = 1$ and $i = n$. In both of these cases we can assume that $x$ does not divide $\mm$.
\begin{description}
\item[Case ($i=1$)]  In this case, we are adding the generator $z_1z_2xy_1$ to $J^2$, our edge ideal of the antipath, i.e. $Q = J^2 : z_1z_2 xy_1$.  Note that $Q\supseteq (y_2, \ldots, y_n, z_2)$.  Hence, if we have a minimal monomial generator $\mm\in Q$ which is not linear, its support must be contained in $\{z_1, y_1\}$.

If $z_1 |\mm$ then $z_1^2 |\mm'$ so $\mm'$ must be of the form $z_1^2 y_jy_k$ with $j,k > 1$.  However, we then have $\mm':z_1z_2xy_1 = z_1 y_j y_k$ which cannot be a minimal generator of $Q$, as both $y_j,y_k \in Q$.

If $y_1 | \mm$ then $y_1^2 | \mm'$ so $\mm'$ must be of the form $y_1^2 y_j z_2$ (for $j > 2$) or $y_1^2 y_j y_k$ (for $j,k > 2$) or $y_1^2 z_2^2$. In these three cases the $\mm'$ are $y_1 y_j$, $y_1 y_j y_k$, and $y_1 z_2$ respectively. However each of these are not minimal, from $y_j,z_2 \in Q$ for $j > 2$.

\item[Case ($i=n$)]

Now we are adding the final generator $z_1z_2xy_n$ to the ideal
		$$I_{z_1z_2xy_n}=J^2+(z_1z_2xy_i:1\leq i\leq n-1).$$
For this, we have $Q = (J^2 + (z_1z_2 xy_j \mid 1 \leq j \leq n-1) ) : (z_1 z_2 x y_n)$ which satisfies $Q\supseteq (y_1, \ldots, y_{n-1}, z_1).$  In this case, if we have a minimal monomial generator $\mm \in Q$ which is not linear, its support must be contained in $\{z_2, y_n\}$.

If $z_2 | \mm$ then $z_2^2 | \mm'$.  The only such $\mm'\in I_{z_1z_2zy_n}$ must be of the form $z_2^2 y_jy_k$ with $j,k < n$.  However, we then have $\mm':z_1z_2xy_1 = z_2 y_j y_k$ which is not a minimal generator as $y_j,y_k \in Q$.

Similarly, if $y_n |\mm$ then $y_n^2 |\mm'$.  All such $\mm'\in I_{z_1z_2zy_n}$ are of one of the following three forms:
\begin{enumerate}[{\bf (i)}]
\item $y_n^2 y_j z_1$ (for some $j < n-1$)
\item $y_n^2 y_j y_k$ (for some $j,k < n-1$)
\item $y_n^2 z_1^2$.
\end{enumerate}
In these three cases the $\mm=\mm':M$ is
\begin{enumerate}[{\bf (i)}]
\item $\mm=y_n^2 y_j z_1:z_1z_2zy_n=y_jy_n$,
\item $\mm=y_n^2 y_j y_k:z_1z_2zy_n=y_j y_ky_n$, and
\item $\mm= y_n^2 z_1^2:z_1z_2zy_n=y_n z_1$ respectively.
\end{enumerate}
However each of these are not minimal as $y_j,z_1 \in Q$ for $j < n-1$.
\end{description}
So our ordering of our generators is a linear quotients ordering through the end of stage (2a).
\subsubsection{Stage (2b):}  The second part of the second stage involves adding monomials $M = xy_iy_jz_2$ to our ideals $I_{M}$ according to the lex order on $(i,j)$.
\begin{align*}
Q &= I_{x y_iy_j z_2}:(x y_i y_j z_2)\\
&=(J^2 + (z_1z_2 xy_j \mid 1 \leq j \leq n) + (x y_{i'} y_{j'} z_2 :  (i',j') >_\lex (i,j) ) : (x y_i y_j z_2)
\end{align*} 
Note the following inclusions hold, via the elements noted.

\begin{itemize}
\item $Q \supseteq (y_k \mid k < j)$ as $y_k =  x y_i y_k z_2 : x y_i y_j z_2$
\item $Q \supseteq (z_1)$ as $z_1 =  x y_i z_1 z_2 : x y_i y_j z_2$
\item $Q \supseteq (z_2)$ when $j \neq n$ as $z_1 =  y_i y_j z_2^2 : x y_i y_j z_2$
\item $Q \supseteq (y_k \mid k > j+ 1)$ as $y_k =  y_i y_j y_k z_2 : x y_i y_j z_2$
\item $Q \supseteq (y_{j+1})$ when $i \neq j$ as $y_{j+1} =  y_i y_j y_{j+1} z_2 : x y_i y_j z_2$
\item $Q \supseteq (y_j)$ when $i \leq j-2$ and $j \neq n$ as $y_j =  y_i y_j^2 z_2 : x y_i y_j z_2$
\end{itemize}
\noindent Taken together for $M=xy_iy_jz_2$ this gives
	\[	Q \supseteq 
			\begin{cases}
			(y_1, \ldots, y_n, z_1, z_2) & j \neq n, i < j-1 \\
			(y_1, \ldots, y_{j-1}, y_{j+1}, \ldots, y_n, z_1, z_2) & j \neq n, i+1 = j \\
			(y_1, \ldots, y_{j-1}, y_{j+2}, \ldots, y_n, z_1, z_2) & j \neq n, i = j  \\
			(y_1, \ldots, y_{n-1}, z_1) & j = n.
			\end{cases}
	\]
Assume $\mm\in Q$ is a minimal monomial generator that is not linear. That is $\mm = \mm' : x y_i y_j z_2$ for some $\mm'$ before $x y_i y_j z_2$. As $\mm$ is minimal, its support cannot contain any of the variables in $Q$. Also if $x$ were to be in $\supp(\mm)$ then $x^2$ would divide $\mm'$.   As no there is no such $\mm'\in I_{M}$ before $x y_iy_kz_2$, we have $x\not|\mm$. Thus the support of $\mm$ satisfies
	\[	\supp(\mm) \subseteq \begin{cases}
					\emptyset &  j \neq n, i < j-1 \\
					\{y_j \} & j \neq n, i+1 =j \\
					\{y_j, y_{j+1} \} & j \neq n, i = j \\
					\{y_n, z_2 \} & j = n
				\end{cases}
	\]

\begin{description}
\item[Case ($j \neq n, i < j-1$)]  There is nothing to check as $x$ does not divide $\mm$ and all other variables are in $Q$.

\item[Case ($j \neq n, i + 1 = j$)]  In this case $\mm$ must be a power of $y_j$. As $\mm$ is not linear, $y_j^2 | \mm$ and hence $y_j^3 | \mm'$.  However none of the generators of $I(G)^2$ are divisible by $y_j^3$.

\item[Case ($j \neq n, i = j$)]  In this case $\supp(\mm)\subseteq \{ y_j, y_{j+1} \}$.  As $\mm$ is not linear, we have one of the following must hold:
\begin{enumerate}[{\bf (i)}]
\item $y_j^2 |\mm$
\item $y_j y_{j+1} |\mm$
\item $y_{j+1}^2 |\mm$.
\end{enumerate}
In these three cases respectively we must then have
\begin{enumerate}[{\bf (i)}]
\item $y_j^3 |\mm'$
\item $y_j^2 y_{j+1}|\mm'$
\item $\mm'\in\bigl\{y_j^2 y_{j+1}^2, y_jy_{j+1}^3, y_{j+1}^4, xy_j y_{j+1}^2, xy_{j+1}^3, z_2y_jy_{j+1}^2, z_2y_{j+1}^3, x z_2 y_{j+1}^2\bigr\}.$
\end{enumerate}
Case {\bf (i)} cannot happen, as $y_j^3$ does not divide any generator of $I(G)^2$.  Similarly, in case {\bf (ii)}, $y_j^2 y_{j+1} | \mm'$ which would require $y_j y_{j+1} \in I(G)$, which is not a generator of the edge ideal of the anticycle.

Finally, in case {\bf (iii)} all degree 4 monomials divisible by $y_{j+1}^2$ have been enumerated as possible $\mm'$.  None of these are generators of $I(G)^2$ except for $\mm'=xz_2 y_{j+1}^2$.  This however occurs later in our order.
\item[Case ($j = n$)]
		In this case $\supp(\mm) \subseteq \{ y_n, z_2 \}$. As $\mm$ is not linear, one of $y_n^2$, $y_n z_2$ and $z_2^2$ divide $\mm$.  If $y_n^2$ or $z_2^2$ divide $\mm$ then $y_n^3$ or $z_2^3$ divide $\mm'$. However no generator of $I(G)^2$ is divisible by a cube of a variable. If $y_n z_2 | \mm$ then $\mm' = y_n^2 z_2^2$ which is not a generator of $I(G)^2$.
\end{description}
\subsubsection{Stage (2c):}  Showing that this part of the ordering is a linear quotients ordering can be done using its symmetry with Stage (2b).  We wish to show that all $Q$ such that
\begin{align*}
Q
&= I_{x y_i y_j z_1}:(x y_iy_j z_1)\\
&=\biggl(J^2 
	+ \bigl(z_1z_2 xy_j \mid 1 \leq j \leq n\bigr) 
	+ \bigl(x y_k y_l z_2 \mid 1 \leq k \leq l \leq n, k < n\bigr)\\
&\;\;\;\;\;+ \bigl( x y_k y_l z_1 \mid (k,l) <_{\lex'}(i,j)\bigr) \biggr) : (x y_iy_j z_1)
\end{align*}
are again generated by variables.  We first show that $Q'$ is generated by variables, for
\begin{equation*}
Q'=\biggl(J^2 
	+ \bigl(z_1z_2 xy_j \mid 1 \leq j \leq n\bigr)+ \bigl( x y_k y_l z_1 \mid (k,l) <_{\lex'}(i,j)\bigr) \biggr) : (x y_iy_j z_1),
\end{equation*}
where the $<_{\lex'}$ denotes the lex ordering on $y_i$ with the variables in reverse order from the $<_{\lex}$ used in Stage (2b).

Via symmetry with Stage (2b), this $Q'$ must have linear quotients via an identical proof.  From this, we see
	\[	Q' =  
			\begin{cases}
			(y_1, \ldots, y_n, z_1, z_2) & j \neq n, j < i-1 \\
			(y_1, \ldots, y_{j-1}, y_{j+1}, \ldots, y_n, z_1, z_2) & i \neq 1, j+1 = i \\
			(y_1, \ldots, y_{j-1}, y_{j+2}, \ldots, y_n, z_1, z_2) & i \neq 1, i = j  \\
			(y_2, \ldots, y_{n}, z_2) & i = 1.
			\end{cases}
	\]

Clearly $Q' \subset Q$.  We note that $Q$ and $Q'$ only differ by a colon ideal of the form
$$\bigl(x y_k y_l z_2 \mid 1 \leq k \leq l \leq n, k < n\bigr):(xy_iy_jz_1).$$
The generators of $Q$ which are not in $Q'$ are of the form $x y_k y_l z_2 : x y_i y_j z_1$ and hence all must divisible by $z_2$.

Since $z_2 \in Q'$ in all cases, we see that $Q$ is generated by variables for all monomials $M$ added in this stage.
\subsubsection{Stage (2d):}  For the final case of Stage 2, we add all monomials in $J\cdot K$ of the form $\mm=x y_i y_j y_k$ ordered via $\lex$ with $y_1>y_2>\cdots y_n$.  Our colon ideals then are of the form
\begin{align*}
Q&=I_{x y_i y_j y_k}:(x y_i y_j y_k)\\
&= \biggl(J^2+\bigl(z_1z_2 xy_j \mid 1 \leq j \leq n\bigr)\\
&+ \bigl(x y_k y_l z_2 \mid 1 \leq k \leq l \leq n, k < n\bigr) + \bigl(x y_k y_l z_1 \mid 1 \leq k \leq l \leq n, 1 < l\bigr)\\
&+ \bigl(x y_{i'} y_{j'} y_{k'} \mid 1 \leq i' \leq j' \leq k' \leq n, i'+2 \leq k', (i',j',k') >_\lex (i,j,k)\bigr)\biggr): (x y_i y_j y_k).
\end{align*}
The last set of generators in $I_{xy_iy_jy_k}$ are given by
$$ \bigl(x y_{i'} y_{j'} y_{k'} \mid 1 \leq i' \leq j' \leq k' \leq n, i'+2 \leq k', (i',j',k') >_\lex (i,j,k)\bigr)$$
as the variables can be arranged with indices $i',j',k'$ in increasing order, but $i'+2\leq k'$ as at least one pair of $\{y_{i'},y_{j'},y_{k'}\}$ must be nonadjacent in the anticycle graph.  This forces the given inequality.

Our colon ideals now satisfy the following inclusions, via the elements noted.
\begin{itemize}
\item $Q \supseteq (y_l \mid l < j)$ as $y_l =  x y_i y_l y_k : x y_i y_j y_k$
\item $Q \supseteq (z_2)$ as $z_2 =  x y_i y_k z_2 : x y_i y_j y_k$
\item $Q \supseteq (z_1)$ as $z_1 =  x y_i y_k z_1 : x y_i y_j y_k$
\item $Q \supseteq (y_l \mid l \geq j +2)$ as $y_l = y_i y_j y_k y_l : x y_i y_j y_k$
\item $Q \supseteq (y_{j+1})$ when $i +1 \leq j$ and $j+2 \leq k$ as $y_{j+1} = y_i y_j y_{j+1} y_k: x y_i y_j y_k$
\item $Q \supseteq (y_j)$ when $i + 2 \leq j$ and $j+2 \leq k$ as $y_{j+1} = y_i y_j^2 y_k : x y_i y_j y_k$.
\end{itemize}

Together this gives

	\[	Q \supseteq 
			\begin{cases}
			(y_1, \ldots, y_{j-1}, y_{j+1}, \ldots, y_n, z_1, z_2) & i = j - 1  \text{ and } j + 2 \leq k\\ 
			(y_1, \ldots, y_{j-1}, y_{j+2}, \ldots, y_n, z_1, z_2) & i = j \text{ or } j = k, k-1  \\
			(y_1, \ldots, y_n, z_1, z_2) & \text{otherwise.}
			\end{cases}
	\]

Assume $\mm \in Q$ is a minimal monomial generator that is not linear. That is $\mm = \mm' : x y_i y_j y_k$ for some $\mm'$ before $M=x y_i y_j y_k$. As $\mm$ is minimal, its support cannot contain any of the variables in $Q$. Also if $x |\mm$ then $x^2| \mm'$.  As this does not happen for any $\mm'$ before $x y_iy_jy_k$, $x\not\in\supp{(\mm)}$. Thus the support of $\mm$ satisfies

	\[	\supp(\mm) \subseteq \begin{cases}
					\{y_j \} & i = j - 1  \text{ and } j + 2 \leq k\\ 
					\{y_j, y_{j+1} \} & i = j \text{ or } j = k, k-1  \\
					\emptyset &  \text{otherwise.}
				\end{cases}
	\]

\begin{description}
\item[Case ($i = j - 1  \text{ and } j + 2 \leq k$)]

In this case, $\mm$ must be divisible only by $y_j$ and cannot be linear. Thus $y_j^2 | \mm$ and $y_j^3 | \mm'$ which does not hold for any generator $\mm'\in I(G)^2$.

\item[Case ($i = j \text{ or } j = k, k-1$)]
In this case, $\mm$ has its support contained in $\{y_j, y_{j+1}\}$. As in the previous case, if the support of $\mm$ contains $\{y_j\}$, we obtain a contradiction.

If the support of $\mm$ contains $\{y_{j+1}\}$ and then $\mm'$ must the product of $y_{j+1}^2$ and two of $x, y_i, y_j, y_k$. However, for this to be a generator of $I(G)^2$ the two chosen vertices must both be adjacent to $y_j$. If $i = j$, then $\mm' x y_{j+1}^2 y_k$ is the only possibility, but this comes after $x y_j^2 y_k$ in our ordering.  If $j=k$ or $j=k-1$ then $\mm'= x y_i y_{j+1}^2$ is the only possibility.  This again lies after $M=x y_i y_j y_k$ in the ordering.

\item[Other Cases]
In the other cases, the quotient contains all variables (except $x$, but there is no term divisible by $x^2$ which occurs prior to $M$ in the ordering.)  Hence, $Q$ must be generated by linear terms. \end{description}
\subsubsection{Stage (3a):}
Now we move on to adding those terms in $K^2$, meaning monomials in $I(G)^2$ which came from pairs of edges $xy_i$ and $xy_j$.  Our colon ideals will be of the form:
\begin{align*}
Q &=I_{x^2y_iy_j}:(x^2y_iy_j)\\
&= \biggl(J^2 + \bigl(z_1z_2 xy_j \mid 1 \leq j \leq n\bigr)+ \bigl(x y_k y_l z_2 \mid 1 \leq k \leq l \leq n, k < n\bigr)\\
& + \bigl(x y_k y_l z_1 \mid 1 \leq k \leq l \leq n, 1 < l\bigr)+ \bigl(x y_i y_j y_k \mid 1 \leq i \leq j \leq k \leq n, i+2 \leq k\bigr)\\
&+ \bigl(x^2 y_k y_l \mid 1 \leq k \leq l \leq n, 1 < l, (k,l) >_\lex (i,j)\bigr)\biggr): (x^2 y_i y_j).
\end{align*}
These colon ideals satisfy the following inclusions via the elements noted.
\begin{itemize}
	\item $Q \supseteq (y_1)$ when $j > 3$ as $y_1 = x y_1 y_i y_j : x^2 y_i y_j$ 
	\item $Q \supseteq (y_1)$ when $i > 1$ as $y_1 = x^2 y_1 y_i : x^2 y_i y_j$ 
	\item $Q \supseteq (y_k \mid 1 < k < j)$ as $y_k = x^2 y_i y_k : x^2 y_i y_j$ 
	\item $Q \supseteq (y_k \mid i + 2 \leq k \leq n)$ as $y_k =  x y_i y_j y_k : x^2 y_i y_j$ 
	\item $Q \supseteq (z_2)$ when $i \neq n$ as $z_2 =  x y_i y_j z_2 : x^2 y_i y_j$
	\item $Q \supseteq (z_1)$ when $j \neq 1$ as $z_1 =  x y_i y_j z_1 : x^2 y_i y_j$
\end{itemize}

	Together this gives
	\[	Q \supseteq 
			\begin{cases}
			(y_3, \ldots, y_n, z_1, z_2) & i=1, j=2\\
			(y_1, \ldots, y_n, z_1, z_2) & i + 2 \leq j\\ 
			(y_1, \ldots, y_{j-1}, y_{j+1}, \ldots, y_n, z_1, z_2) & 1 < i = j -1\\
			(y_1, \ldots, y_{j-1}, y_{j+2}, \ldots, y_n, z_1, z_2) & 1 < i = j < n\\
			(y_1, \ldots, y_{n-1}, z_1) & {i = j = n}.
			\end{cases}
	\]
Assume $\mm \in Q$ is a minimal monomial generator that is not linear. That is $\mm = \mm' : x^2 y_i y_j$ for some $\mm'$ before $M=x^2 y_i y_j$. Again, as $\mm$ is minimal its support cannot contain any of the variables in $Q$. Also if $x | \mm$ then $x^3|\mm'$ which does not happen for any $\mm'\in I(G)^2$.  Thus the support of $\mm$ satisfies

	\[	\supp(\mm) \subseteq \begin{cases}
					\{y_1, y_2\} & i = 1, j = 2 \\
					\emptyset & i + 2 \leq j \\
					\{y_j \} & i = j - 1 \\
					\{y_j, y_{j+1} \} & 1 < i = j < n \\
					\{y_n, z_2 \} &  i = j = n.
				\end{cases}
	\]
We examine each of these cases individually.
\begin{description}
\item[Case ($i=1, j=2$)]
In this case $\mm$ is divisible by one of $y_1^2, y_1 y_2, y_2^2$ and hence $\mm'$ is divisible by $y_1^3, y_1^2 y_2^2, y_2^3$.  None of these can hold for $\mm'$ a generator of $I(G)^2$.
\item[Case ($i + 2 \leq j $)]
There is nothing to check as $x$ does not divide $\mm'$ and all other variables are in $Q$.
\item[Case ($i = j - 1$)]
In this case $\mm$ must be a power of $y_j$. As $\mm$ is not linear, $y_j^2 |\mm'$ and hence $y_j^3 |\mm$.  No generators of $I(G)^2$ are divisible by $y_j^3$ (or any third power of a variable.)
\item[Case ($1 < i = j < n$)]
In this case $\mm$ is divisible by one of $y_j^2$, $y_j y_{j+1}$ or $y_{j+1}^2$.  If $\mm'$ is to appear before $x^2 y_i y_j$ in our list, it cannot be $x^2 y_j^2, x^2 y_j y_{j+1},$ nor $x^2 y_{j+1}^2$.  As $i=j$, the remaining possibilities for $\mm$ are $xy_j^3, xy_j^2 y_{j+1}, x y_j y_{j+1}^2$ or a monomial of degree four in $y_j$ and $y_{j+1}$. However, none of these are generators of $I(G)^2$.
\item[Case ($i = j = n$)]
In this case $\mm$ is divisible by one of $y_n^2, y_n z_2, z_2^2$.  So $\mm'$ is divisible by one of $y_n^4$, $y_n^3z_2$, $z_2^2$.
There are no $\mm' \in I(G)^2$ such that the first two hold.  For the last, if $z_2^2 | \mm'$ and $y_n$ does not divide $\mm$ then $\mm'$ must be one of $z_2^4, z_2^3 x, z_2^3 y_1, z_2^2 x^2, z_2^2 x y_n, z_2^2 y_n^2$.  None of these are in $I(G)^2$.
\end{description}
From this, we see that $I(G)^2$ has a linear quotients through Stage (3a).
\subsubsection{Stage (3b):}
Finally, we add our generator $x^2y_1^2$ to our ideal $I_{x^2y_1^2}$.  We only need to check that for this one remaining generator, the following colon ideal is generated by variables:
\begin{align*}
Q &= I_{x^2y_1^2}:(x^2y_1^2)\\
&=\biggl(J^2 + \bigl(z_1z_2 xy_j \mid 1 \leq j \leq n\bigr) + \bigl(x y_k y_l z_2 \mid 1 \leq k \leq l \leq n, k < n\bigr) \\
&+ \bigl(x y_k y_l z_1 \mid 1 \leq k \leq l \leq n, 1 < l\bigr) + \bigl(x y_i y_j y_k \mid 1 \leq i \leq j \leq k \leq n, i+2 \leq k\bigr)\\
&+ \bigl(x^2 y_k y_l \mid 1 \leq k \leq l \leq n,  1< l\bigr)\biggr): (x^2 y_1^2).
\end{align*}
We have the following inclusions by the elements noted:
\begin{itemize}
	\item $Q \supseteq (y_k \mid 1 < k \leq n)$ as $y_k = x^2 y_1 y_k : x^2 y_1^2$ 
	\item $Q \supseteq (z_2)$ when $i \neq n$ as $z_2 =  x y_1^2 z_2 : x^2 y_1^2$.
\end{itemize}
This gives us that our colon ideal satisfies $Q \supseteq (y_2, \ldots, y_n, z_2)$.

So, if $\mm \in Q$ is a minimal non-linear monomial, then $\supp(\mm) \subseteq \{y_1, z_1\}$ and $\mm = \mm':x^2 y_1^2$ for some $\mm \in I(G)^2$ before $x^2 y_1^2$.  If $y_1 | \mm$ then $\mm'$ must be divisible by $y_1^3$.  There is no such $\mm' \in I(G)^2$. Thus $\supp(\mm) = \{z_1\}$.

Since by assumption, $\mm$ is not linear, $z_1^2 | \mm$. Thus, $z_1^2 | \mm'$ and the other variables dividing $\mm'$ can only be $z_1, x$ or $y_1$.  There is no way to form a generator of $I(G)^2$ using only these variables as $y_1$ and $x$ and $z_1$ are not adjacent to $z_1^2$.  Hence, $Q=(y_2, \ldots, y_n, z_2)$.

So this provides a linear quotients ordering on $I(G)^2$.
\end{proof}

\section{Future Research}
For higher powers of the edge ideal $I(A_n)^k$ of the anticycle, it is still unknown if all powers have a linear resolution, much less linear quotients.  Construction of linear quotient orderings on $I(A_n)^k$ would accomplish this.
\begin{question}  Does $I(A_n)^k$ have linear quotients for $k\geq 3$?
\end{question}

We produced an ordering above on $I(A_n)^2$ by decomposing $A_n$ into complementary subgraphs $P_{n-1}$ and $A_n\setminus P_{n-1}$.  While this order is nonunique, ordering the edges of $I(A_n)^2$ by decomposing the graph into the complementary subgraphs $H$ and $G\setminus H$, then considering pairs of edges as appropriate, seems to produce linear quotients orderings with the clearest descriptions.  Extending this order to $I(G)^k$ in a similar fashion has proven fairly difficult, even in the case of $I(G)^3$, but would be a natural next step after Theorem~\ref{thm:mainanticycletheorem}.

A problem of more general interest is to complete Theorem~\ref{thm:HHZ} of Herzog, Hibi and Zheng by answering the following question:
\begin{question}  Let $G$ be the complement of a chordal graph.  Does $I(G)^k$ have linear quotients for $k\geq 2$?
\end{question}

We might also ask for a description of all edge ideals whose powers eventually have linear resolutions.
\begin{question}\label{ques:classesofgraphs}  Can we exhibit classes of graphs $G$ such that for all sufficiently large $k$,
\begin{enumerate}[(i)]
\item\label{ques:subclass1} $I(G)^k$ has a linear resolution, or
\item\label{ques:subclass2} $I(G)^k$ has linear quotients?
\end{enumerate}
\end{question}
In \cite{NP2009}, it was conjectured that graphs satistfying Question~\ref{ques:classesofgraphs}(\ref{ques:subclass1}) are precisely those graphs $G$ with a $C_4$-free complement.  General conditions for the second class however remain open.  It appears that anticycles $A_n$ form such a class, but we wish to find more general conditions for the powers of an edge ideal of a graph to stabilize on linear quotients.
\medskip

\textbf{Acknowledgements.} 
We would like to thank Irena Peeva and Eran Nevo for getting us interested in this topic and Adam Van Tuyl for many stimulating conversations.
The first author would like to thank his advisor, Sara Faridi, for her direction and enouragement on this project.
The second author would like to thank her advisor, Mike Stillman, for his detailed commentary on this paper, greatly improving the exposition.

\bibliographystyle{amsalpha}
\bibliography{linearquotientswriteup}

\providecommand{\bysame}{\leavevmode\hbox to3em{\hrulefill}\thinspace}
\providecommand{\MR}{\relax\ifhmode\unskip\space\fi MR }
\providecommand{\MRhref}[2]{%
  \href{http://www.ams.org/mathscinet-getitem?mr=#1}{#2}
}
\providecommand{\href}[2]{#2}
\begin{thebibliography}{HVT08}

\bibitem[Con06]{MR2184787}
Aldo Conca, \emph{Regularity jumps for powers of ideals}, Commutative algebra,
  Lect. Notes Pure Appl. Math., vol. 244, Chapman \& Hall/CRC, Boca Raton, FL,
  2006, pp.~21--32. \MR{2184787 (2007c:13017)}

\bibitem[Fr{\"o}90]{MR1171260}
Ralf Fr{\"o}berg, \emph{On {S}tanley-{R}eisner rings}, Topics in algebra,
  {P}art 2 ({W}arsaw, 1988), Banach Center Publ., vol.~26, PWN, Warsaw, 1990,
  pp.~57--70. \MR{1171260 (93f:13009)}

\bibitem[HHZ04]{MR2091479}
J{\"u}rgen Herzog, Takayuki Hibi, and Xinxian Zheng, \emph{Monomial ideals
  whose powers have a linear resolution}, Math. Scand. \textbf{95} (2004),
  no.~1, 23--32. \MR{2091479 (2005f:13012)}

\bibitem[HVT07]{MR2301246}
Huy~T{\`a}i H{\`a} and Adam Van~Tuyl, \emph{Splittable ideals and the
  resolutions of monomial ideals}, J. Algebra \textbf{309} (2007), no.~1,
  405--425. \MR{2301246 (2008a:13016)}

\bibitem[HVT08]{Ha:2008:MIE:1344768.1344814}
Huy~T\`{a}i H\`{a} and Adam Van~Tuyl, \emph{Monomial ideals, edge ideals of
  hypergraphs, and their graded betti numbers}, J. Algebraic Comb. \textbf{27}
  (2008), 215--245.

\bibitem[MV10]{2010arXiv1012.5329M}
Susan Morey and Rafael~H. Villarreal, \emph{{Edge ideals: algebraic and
  combinatorial properties}}, arXiv:1012.5329v3 [math.AC] (2010).

\bibitem[Nev11]{MR2739498}
Eran Nevo, \emph{Regularity of edge ideals of {$C_4$}-free graphs via the
  topology of the lcm-lattice}, J. Combin. Theory Ser. A \textbf{118} (2011),
  no.~2, 491--501. \MR{2739498}

\bibitem[NP09]{NP2009}
Eran Nevo and Irena Peeva, \emph{Linear resolutions of powers of edge ideals},
  preprint (2009).

\bibitem[Vil90]{MR1031197}
Rafael~H. Villarreal, \emph{Cohen-{M}acaulay graphs}, Manuscripta Math.
  \textbf{66} (1990), no.~3, 277--293. \MR{1031197 (91b:13031)}

\end{thebibliography}
\end{document}